\documentclass{article}
\usepackage[a4paper, total={5.5in, 8in}]{geometry}
\usepackage[utf8]{inputenc}
\usepackage{amsmath}
\usepackage{amsthm}
\usepackage{amssymb}
\usepackage{graphicx}
\usepackage{xcolor}
\usepackage[sort]{natbib}
\usepackage{authblk}

\theoremstyle{definition}
\newtheorem{theorem}{Theorem}[section]

\newtheorem{lemma}[theorem]{Lemma}
\newtheorem{proposition}[theorem]{Proposition}
\newtheorem{definition}{Definition}[section]
\newtheorem{example}{Example}[section]
\newtheorem{conjecture}{Conjecture}[section]

\newcommand{\sym}[1]{\mathfrak{S}_{#1}}
\newcommand{\atoms}[1]{R_{\bullet}(#1)}
\newcommand{\red}[1]{\textcolor{red}{#1}}

\newcommand{\green}[1]{\textcolor{green}{#1}}
\newcommand{\uwidehat}[1]{%
	\mathpalette\douwidehat{#1}%
}

\makeatletter
\newcommand{\douwidehat}[2]{%
	\sbox0{$\m@th#1\widehat{\hphantom{#2}}$}%
	\sbox2{$\m@th#1x$}
	\sbox4{$\m@th#1#2$}
	\dimen0=\ht0
	\advance\dimen0 -.8\ht2
	\dimen2=\dp4
	\rlap{%
		\raisebox{\dimexpr\dimen0-\dimen2}{%
			\scalebox{1}[-1]{\box0}%
		}%
	}%
	{#2}%
}
\makeatother

\newcommand{\tupp}[2]{\uwidehat{#1}\,\widehat{#2}}
\newcommand{\tdownn}[2]{\widehat{#1}\,\uwidehat{#2}}
\newcommand{\unn}[1]{\uwidehat{#1}}
\newcommand{\ovv}[1]{\widehat{#1}}

\newcommand{\p}[1]{\text{p}(#1)}
\newcommand{\vv}[1]{\text{v}(#1)}
\newcommand{\pv}[1]{\text{pv}(#1)}

\renewcommand{\l}{\ell}

\title{Maximum number of one-element commutation classes of a permutation}
\author[1]{Ricardo Mamede}
\author[1]{José Luis Santos}
\author[1]{Diogo Soares}

\affil[1]{University of Coimbra, CMUC, Department of Mathematics}

\begin{document}
	\maketitle

\begin{abstract}
	In this paper, we provide an upper bound for the number of one-element commutation classes of a permutation, that is, the number of reduced words in which no commutation can be applied. Using this upper bound, we prove a conjecture that relates the number of reduced words with the number of commutation classes of a permutation.
\end{abstract}

\section{Introduction}
Given a positive integer $n\geq 1$, let $\sym{n+1}$ denote the \textit{symmetric group} on $n+1$ elements formed by all permutations of the set $[n+1]:=\{1,2,\ldots,n+1\}$ with composition (read from the right) as group operation. To denote a permutation $\sigma \in \sym{n+1}$, we will use the \textit{one-line notation} $\sigma:=(\sigma(1),\sigma(2),\ldots,\sigma(n+1))$, as well as the classic \textit{cycle notation}, where cycles are written in round brackets and the images separated by spaces. For instance, $\sigma=(3,4,1,5,2)\in \sym{5}$ can be expressed in cycle notation as $\sigma=(1\ 3)(2\ 4\ 5)$. \par  
The symmetric group is a famous example of a more general family of groups called \textit{Coxeter groups}, which are groups generated by a set of involutions (elements of order 2) subject to some specific relations (see \cite{Brenti,Humphreys_1990}). In this case, the group $\sym{n+1}$ is generated by the involutions $\{s_i: i\in [n]\}$, where $s_i:=(1,\ldots
,i+1,i,\ldots,n+1)=(i\ i+1)$, which satisfy the following relations: 
\begin{align}
s_is_j&=s_js_i,\;\text{if }|i-j|>1 \label{commrel}\\
s_{i}s_{i+1}s_i&=s_{i+1}s_{i}s_{i+1},\;i\in [n-1] \label{braidrel},
\end{align}
with \eqref{commrel} being called a \textit{commutation}, and \eqref{braidrel} a \textit{braid relation}. \par 
Every permutation $\sigma$ can be written as a product $s_{i_1}s_{i_2}\cdots s_{i_l}$ of generators, with $i_j\in [n]$ for all $j\in [l]$. When $l$ is minimal, we say that $s_{i_1}s_{i_2}\cdots s_{i_l}$ is a \textit{reduced decomposition}, and $i_1i_2\cdots i_l$ a \textit{reduced word} of $\sigma$. The integer $l:=\l(\sigma)$ is the \textit{length} of $\sigma$, and it corresponds to the number of \textit{inversion pairs} of $\sigma$, that is, 
$$\l(\sigma)=|\{(\sigma(i),\sigma(j)):i<j,\sigma(i)>\sigma(j)\}|.$$
We denote by $R(\sigma)$ the set of reduced words of $\sigma$. As an example, if $\sigma=(3,4,2,1)$, then $\l(\sigma)=5$ and $R(\sigma)=~\{12132,12312,21232,21323,23123\}$, which corresponds to all possible sequences of indices of the generators for a reduced decomposition of $\sigma$. 
A well-known result of Tits \cite{TITS1969175} states that two reduced words for the same permutation differ by a finite sequence of commutations and/or braid relations. \par
We can define an equivalence relation $\sim$ on the set $R(\sigma)$ by setting $\textbf{a}\sim \textbf{b}$ if these words differ by a sequence of commutations. The equivalence classes generated by $\sim$ are the \textit{commutation classes} of $\sigma$, being the commutation class of $\textbf{a}\in R(\sigma)$ denoted by $[\textbf{a}]$.  The set of all commutation classes of $\sigma$ is denoted by $C(\sigma)$. For $\sigma=(3,4,2,1)$, we have three distinct commutation classes: $[12132]=\{12132,12312\}$, $[21232]=\{21232\}$ and $[21323]=\{21323,23123\}$. These objects were already considered by some authors and are related to other topics in Discrete Mathematics such as Higher Bruhat orders, pseudoline arrangements and rhombic tilings of polygons (see \cite{comutClass2,ELNITSKY1997193,Elder,BEDARD1999483,MAMEDE2022113055,FELSNER200167}). \par 
In general, a commutation class may contain several reduced words. For instance, the number of reduced words inside the commutation class of the word $1\cdot21\cdot 321\cdots n(n-1)\cdots 1$ is equinumerous with the number of shifted standard staircase tableaux \cite{SCHILLING201715}, all of these words being reduced words for the longest permutation $w_0:=(n+1,n,\ldots,1)$. However, it can happen an extreme situation in which a reduced word is its entire commutation class, that is, a word $\textbf{a}\in R(\sigma)$ satisfying $[\textbf{a}]=\{\textbf{a}\}$. A reduced word with such a property is called \textit{one-element commutation class}, and we denote by $\atoms{\sigma}$ the set of one-element commutation classes of $\sigma$.  For instance, the word $\textbf{a}=21232$ is the only one-element commutation class of $\sigma=(3,4,2,1)$. In the case of the longest element, we have $|\atoms{w_0}|=4$ for all $n\geq 3$, a result first obtained in \cite{comutClass} and then recovered by Tenner \cite{Tenner}. This enumeration was later generalized for an arbitrary involution by the authors \cite{OEC}, showing that the number of one-element commutation classes of an involution is either $0$, $1$, $2$ or $4$. An immediate consequence is that an involution cannot have more than four one-element commutation classes. Could it be possible for a general permutation to have more than four such words? The main result of this paper states that such a permutation cannot exist.

\begingroup
\renewcommand\thetheorem{A}
\begin{theorem}[Theorem \ref{bound}]
	Let $n\geq 1$. Then, for all $\sigma\in \sym{n+1}$ we have $|\atoms{\sigma}|\leq 4$.  
	\label{thm1}
\end{theorem}
\endgroup

This result is interesting just by itself as it shows that the number of words for which we can't apply any commutation is very limited. Moreover, it can also be used to prove a conjecture proposed in \cite{Elder} that relates the cardinalities of $R(\sigma)$ and $C(\sigma)$.
\begingroup
\renewcommand\thetheorem{B}
\begin{theorem}[Theorem \ref{conjecture}]
	For all $\sigma \in \sym{n+1}$, we have $\displaystyle|C(\sigma)|\leq\frac{1}{2}|R(\sigma)|+1$. 
	\label{thm2}
\end{theorem}
\endgroup
This paper is organized as follows. Section~2 reviews the terminology from \cite{Tenner,OEC} relevant to one-element commutation classes, along with auxiliary results. In Section~3, we prove Theorem~\ref{thm1}. The key idea is to partition the set of permutations into two disjoint subsets according to a property of their one-element commutation classes, and then apply an inductive argument. Using the upper bound for $|\atoms{\sigma}|$, Section~4 establishes Theorem~\ref{thm2}. Finally, Section~5 presents some open problems and possible directions for future research.

\section{Preliminaries}
Given a set $S$, denote by $S^*$ the free monoid generated by $S$ consisting of all finite words with letters in $S$, with concatenation as multiplication and the empty word as identity element. For our purposes, $S$ will be either the set $[n]$ or $[i,j]:=\{i,i+1,\ldots j\}$, with $i\leq j$ positive integers. By convention, we consider $[i,j]$ to be the empty set if $i>j$. A \textit{subword} of a word $\textbf{a}=i_1i_2\ldots i_l$ is a word obtained from $\textbf{a}$ by deleting some of its letters, while a \textit{factor} is a subword formed by consecutive letters in $\textbf{a}$. \par 

For a reduced word to be a one-element commutation class, some properties must be satisfied. A basic requirement follows from the fact that such words contain no available commutations. Since we can commute a pair of letters in a reduced word of a permutation whenever there is a factor $ij$ with $|i-j|>1$, we have that $\atoms{\sigma}$ is the subset of $R(\sigma)$ formed by the reduced words where every factor of length 2 is of the form $i(i+1)$ or $(i+1)i$. A word with this property is called a \textit{word formed by consecutive integers}. For instance, the word $\textbf{a}=2345432123456765434$ is one such example. Additionally, one-element commutation classes must satisfy some technical properties, which will be important throughout the paper.

\begin{definition}[\cite{OEC,Tenner}]
	Let $\textbf{a} \in [n]^*$ be a word. The endpoints of $\textbf{a}$ are its leftmost and
	rightmost letters. A \emph{pinnacle} of $\textbf{a}$ is a letter that is larger than its immediate neighbor(s),
	and a \emph{vale} is a letter that is smaller than its immediate neighbor(s). We call pinnacles
	and vales the \textit{spikes} of \textbf{a}.
	Write $\p{\textbf{a}}$ for the
	subword of pinnacles of $\textbf{a}$, and $\vv{\textbf{a}}$ for the subword of vales. The subword of pinnacles and vales will be written as $\pv{\textbf{a}}$. 
\end{definition}
For the word $\textbf{a}=\green{2}34\red{5}432\green{1}23456\red{7}654\green{3}\red{4}$ (with vales in green and pinnacles in red), we have $\vv{\textbf{a}}=\green{213}$, $\p{\textbf{a}}=\red{574}$ and $\pv{\textbf{a}}=\green{2}\red{5}\green{1}\red{7}\green{3}\red{4}$. Note that the endpoints are always considered spikes of a word. \par 
A visual interpretation of words formed by consecutive integers can be given as follows. 
\begin{definition}[\cite{OEC}]
	Let $\textbf{a} \in [n]^*$ be a word formed by consecutive integers with $\pv{\textbf{a}}=i_1\cdots i_r$. The \emph{line diagram} of $\textbf{a}$ is formed by the set of points $(j,i_j) \in [r]\times[n]$, where there is a line segment connecting each pair $(j,i_j)$ and $(j+1,i_{j+1})$, for all $j \in [r-1]$.  
\end{definition} 
The line diagram of the word above is represented in Figure \ref{fig:ex_linediag}. Notice that each factor $ij$ of $\pv{\textbf{a}}$ corresponds in $\textbf{a}$ to the factor $i(i+1)\cdots (j-1)j$ if $i<j$, or $i(i-1)\cdots (j+1)j$ if $i>j$. These factors are the \textit{segments} of $\textbf{a}$, and they are encoded by the line segments of its line diagram. To simplify their writing, it will be useful to write only the endpoints of the segments, and for that we use the notation $\unn{i}$ and $\ovv{j}$ to denote a vale $i$ or to denote a pinnacle $j$, respectively. For our running example, its segments are $\unn{2}\ovv{5}$, $\ovv{5}\unn{1}$, $\unn{1}\ovv{7}$, $\ovv{7}\unn{3}$ and $\unn{3}\ovv{4}$. Notice that two consecutive segments always share a spike. In this sense, we write $\unn{i}\,\,\ovv{j}\,\unn{k}$ (resp. $\ovv{i}\,\unn{j}\,\ovv{k}$), to represent the word $\green{i}(i+1)\cdots (j-1)\red{j}(j-1)\cdots (k+1)\green{k}$ (resp. $\red{i}(i-1)\cdots (j+1)\green{j}(j+1)\cdots (k-1)\red{k}$).\par 
\begin{figure}[h]
	\centering
	\includegraphics[scale=0.7]{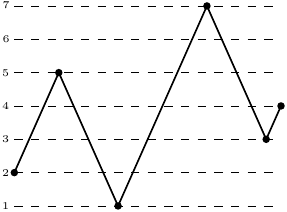}
	\caption{Line diagram of the word $2345432123456765434=\unn{2}\,\ovv{5}\,\unn{1}\,\ovv{7}\,\unn{3}\,\ovv{4}$.}
	\label{fig:ex_linediag}
\end{figure}
It turns out that $\p{\textbf{a}}$, $\vv{\textbf{a}}$ and $\pv{\textbf{a}}$ must satisfy some properties in order to encode a one-element commutation class.


\begin{definition}[\cite{Tenner}]
	Let $\textbf{a}=i_1\cdots i_l \in [n]^*$ be a word. If there exist $j$ and $k$ such that
	$1 \leq j \leq k \leq  l$ and
	$$i_1 < \cdots < i_j = i_k > \cdots > i_l,$$
	then $\textbf{a}$ is a \emph{wedge}. If
	$$i_1 > \cdots > i_j = i_k < \cdots < i_l,$$ then $\textbf{a}$ is a \emph{vee}. If $j = k$, then that wedge or vee is strict.   \label{def_vee}
\end{definition}

Then, we have following.

\begin{theorem}[\cite{Tenner}]
	For any $\sigma \in \sym{n+1}$, if $\textbf{a} \in \atoms{\sigma}$ then:
	\begin{enumerate}
		\item $\p{\textbf{a}}$ is a wedge,
		\item $\vv{\textbf{a}}$ is a vee,
		\item $\p{\textbf{a}}$ and/or $\vv{\textbf{a}}$ is strict,
		\item the minimum and maximum values of $\pv{\textbf{a}}$ appear consecutively and,
		\item if $\p{\textbf{a}}$ (or $\vv{\textbf{a}}$) has more than one occurrence of an integer $i$, then one of those $i$'s is an endpoint of $\textbf{a}$. 
	\end{enumerate}
	\label{TheoremTenner}
\end{theorem}
The previous result provides a set of necessary conditions for a word formed by consecutive integers to be reduced. However, these conditions are not sufficient to fully characterize such words, as there are words formed by consecutive integers satisfying all conditions of Theorem \ref{TheoremTenner} which are not reduced. The full characterization was obtained in \cite{OEC}, and it corresponds to the words formed by consecutive integers that ``avoid" some factors.

\begin{definition}[\cite{OEC}]
	A word \textbf{a} formed by consecutive integers in $[n]^*$ is said to have a factor with \textit{repeated segments} if it contains a factor
	$\tupp{i}{j}\,\textbf{b}\, \tupp{i}{j}$ or $\tdownn{j}{i}\, \textbf{b}\, \tdownn{j}{i}$, for some $i<j$ and $\textbf{b}\in[n]^*$. \label{def:repseg}
\end{definition}
In other words, it contains a factor with at least two occurrences of the same segment. For instance, the word $\textbf{a}=234543212345676543456=\unn{2}\,\ovv{5}\,\unn{1}\,\ovv{7}\,\unn{3}\,\ovv{6}$ contains a factor with repeated segments: the factor $3454321234567654345=\unn{3}\,\ovv{5}\,\unn{1}\,\ovv{7}\,\unn{3}\,\ovv{5}$. Figure \ref{rep_segs} depicts the line diagram of $\textbf{a}$, with the line diagram of $\unn{3}\,\ovv{5}\,\unn{1}\,\ovv{7}\,\unn{3}\,\ovv{5}$ highlighted in red. This factor contains two occurrences of the segment $\tupp{3}{5}$. \par 
\begin{figure}[h]
	\centering
	\includegraphics[scale=0.7]{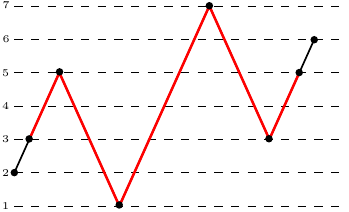}
	\caption{Line diagram of the word $\unn{2}\,\ovv{5}\,\unn{1}\,\ovv{7}\,\unn{3}\,\ovv{6}$ with a factor with repeated segments highlighted in red.}
	\label{rep_segs}
\end{figure}
The previous property can also be checked directly from the line diagram; a word $\textbf{a}$ does not contain factors with repeated segments if and only if its line diagram avoids the shapes depicted in Figure \ref{avoid} (see \cite{OEC}). For instance, the line diagram depicted in Figure \ref{fig:ex_linediag} encodes a word that does not contain factors with repeated segments, since it avoids those shapes
\begin{figure}[h]
	\centering
	\includegraphics[scale=1]{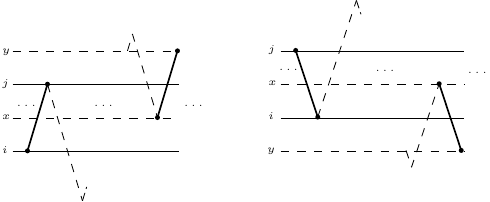}
	\caption{Possible line diagrams of a word containing a factor with repeated segments. On the left (resp. right) line diagram, $x$ satisfies $i\leq x<j\leq y$ (resp. $y\leq i<x\leq j$).}
	\label{avoid}
\end{figure}

\begin{definition}[\cite{OEC}]
	A word \textbf{a} formed by consecutive integers in $[n]^*$ is said to have a factor with \textit{symmetric} segments if it has a factor 
	$\tupp{i}{j}\,\textbf{b}\, \tdownn{j}{i}$ or $\tdownn{j}{i}\, \textbf{b}\, \tupp{i}{j}$, for some $i<j$ where $\textbf{b}\in[n]^*$ contains at least one spike of $\textbf{a}$.  \label{symseg}
\end{definition}
The word $3212343212=\ovv{3}\unn{1}\ovv{4}\unn{1}\ovv{2}$ contains a factor with symmetric segments, namely the factor $212343212=\ovv{2}\unn{1}\ovv{4}\unn{1}\ovv{2}$ (see Figure \ref{sym_segs}). A word that does not contain any factor with symmetric segments is equivalent to a word satisfying condition 5 of Theorem \ref{TheoremTenner} (see \cite{OEC}). Notice that the previous word contains two vales $1$, neither of which is an endpoint. \par 
\begin{figure}[h]
	\centering
	\includegraphics[scale=0.7]{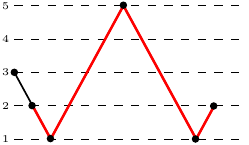}
	\caption{Line diagram of the word $\ovv{3}\unn{1}\ovv{4}\unn{1}\ovv{2}$ with a factor with symmetric segments highlighted in red.}
	\label{sym_segs}
\end{figure}
The characterization of words formed by consecutive integers that are reduced can be given in the following way.
\begin{theorem}[\cite{OEC}]
	Let $\textbf{a}\in[n]^*$ be a word formed by consecutive integers and let $\sigma\in \sym{n+1}$ be the corresponding permutation. Then, $\textbf{a} \in \atoms{\sigma}$ if and only if there is no factor of $\textbf{a}$ with repeated or symmetric segments.  
	\label{TheoremEquiv2}
\end{theorem}

Theorems \ref{TheoremTenner} and \ref{TheoremEquiv2} will be important to prove the main result of this paper.  We end this section with some more auxiliary results.
\begin{lemma}[\cite{TennerRep}]
	Fix a permutation $\sigma \in \sym{n+1}$. The following statements are equivalent:
	\begin{enumerate}
		\item There is a word $\textbf{a}\in R(\sigma)$ containing a letter $i$. 
		\item Every word $\textbf{a}\in R(\sigma)$ contains a letter $i$.
		\item $\{\sigma(1),\ldots,\sigma(i)\}\neq \{1,\ldots,i\}$.
		\item $\{\sigma(i+1),\ldots,\sigma(n+1)\}\neq \{i+1,\ldots,n+1\}$.  
	\end{enumerate}
\label{lemmaTenner}
\end{lemma}
We also recall the classical definition of fixed point of a permutation.
\begin{definition}
	For $i\in [n+1]$ and $\sigma\in \sym{n+1}$, we say that $i$ is a \textit{fixed point} of $\sigma$ if $\sigma(i)=i$. Otherwise, if $\sigma(i)\neq i$ we call $i$ a \textit{non-fixed point}. 
\end{definition}
Every $\sigma \in \sym{n+1}$ that is not the identity permutation contains at least two non-fixed points. This fact is important because the minimum and maximum non-fixed points of a permutation will play an important role in the following section. For instance, if we have \(\sigma =(3,2,5,4,1) \in \sym{5}\), then $1$ and $5$ are, respectively, the minimum and maximum non-fixed points of $\sigma$. \par 
The proof of the following result can be found in \cite{OEC} and it is also a consequence of Lemma \ref{lemmaTenner}.
\begin{lemma}[\cite{OEC}]
	Let $\textbf{a} \in R(\sigma)$ a reduced word for some $\sigma \in \sym{n+1}$ and $i,j\in [n]$ with $i<j$. If $\textbf{a} \in [i,j]^*$, then $\sigma(k)=k$ for all $k \in [n+1]\setminus[i,j+1]$.  \label{lemma:fixedelts} 
\end{lemma}

\begin{lemma}
	Let $\sigma \in \sym{n+1}$ not the identity permutation with $m$ and $M+1$ the minimum and maximum non-fixed points of $\sigma$, respectively, where $1\leq m\leq M\leq n$. If $\textbf{a}\in R(\sigma)$, then $\textbf{a}\in [m,M]^*$. Moreover, $\textbf{a}$ contains at least a letter $m$ and a letter $M$.  
\end{lemma}
\begin{proof}
	From the hypothesis, $\{\sigma(1),\ldots,\sigma(i)\}=[i]$ and $\{\sigma(j),\ldots,\sigma(n+1)\}=[j,n+1]$ for all $i<m$ and $j>M+1$. Thus, $\textbf{a}$ cannot contain letters in the alphabet $[m-1]\cup [M+1,n+1]$ by conditions 3 and 4 of Lemma \ref{lemmaTenner}, implying that $\textbf{a}\in [m,M]^*$. \par 
	To prove that $\mathbf{a}$ must contain the letter $m$, note that the fact that $\{\sigma(1),\ldots,\sigma(m-1)\} = [m-1]$ and $\sigma(m) \neq m$ implies that $\sigma(m) = j > m$. Hence $\{\sigma(1),\ldots,\sigma(m)\} \neq [m]$, and by Lemma~\ref{lemmaTenner}, $\mathbf{a}$ must contain the letter $m$.
	Using a similar reasoning and the fact that $M+1$ is the maximum non-fixed point, one can show that $\textbf{a}$ must also contain a letter $M$.
\end{proof}
For consistency, we will denote by $m$ and $M+1$ the minimum and maximum non-fixed points of a permutation throughout, with $1 \leq m \leq M \leq n$. Notice that if $m = M$, then $\sigma = (m\; m+1)$, which has only one reduced word; the single letter $m$. Since this is a trivial case for the results that follow, we will also  assume $m < M$, unless otherwise stated.

\section{An upper bound for $|\atoms{\sigma}|$}
As we mentioned in the beginning, to prove Theorem \ref{thm1} we will first partition the set of permutations in two disjoints subsets depending on a property of their one-element commutation classes. For that, we need to uncover some structure appearing in these words. 

\begin{lemma}
 Let $\sigma \in \sym{n+1}$ with $m$ and $M+1$ the minimum and maximum non-fixed points of $\sigma$, respectively. If $|\atoms{\sigma}|>0$, then $\sigma(m)=M+1$ or $\sigma(M+1)=m$. Moreover, if $\textbf{a}\in \atoms{\sigma}$, then
 \begin{enumerate}
 	\item $\textbf{a}=\textbf{p}\,\tupp{m}{M}\,\textbf{q}$ for some $\textbf{p}\in [m+1,M]^*$ and $\textbf{q}\in [m,M-1]^*$ if and only if $\sigma(M+1)=m$.
 	\item$\textbf{a}=\textbf{p}\,\tdownn{M}{m}\,\textbf{q}$ for some $\textbf{p}\in [m,M-1]^*$ and $\textbf{q}\in [m+1,M]^*$ if and only if $\sigma(m)=M+1$.  
 \end{enumerate}
\label{lemma:endpoints}
\end{lemma}
\begin{proof}
Suppose that $\textbf{a}\in \atoms{\sigma}$, which exists by assumption. From the previous lemma, $\textbf{a}\in [m,M]^*$ and it contains at least a letter $m$ and a letter $M$. Clearly, those letters are the minimum and maximum values of $\textbf{a}$, and each letter $m$ (resp. $M$) is a vale (resp. pinnacle) of $\textbf{a}$. Now suppose there is an occurrence of a letter $m$ proceeded by a letter $M$ in $\textbf{a}$. From condition 4 of Theorem \ref{TheoremTenner}, since letters $m$ and $M$ must appear consecutively in $\pv{\textbf{a}}$, the leftmost $m$ and the rightmost $M$ of $\textbf{a}$ must form a segment $\tupp{m}{M}$. Thus, we can write $\textbf{a}$ as $$\textbf{a}=\textbf{p}\,\tupp{m}{M}\,\textbf{q},$$ for some words $\textbf{p}\in [m+1,M]^*$ and $\textbf{q}\in [m,M-1]^*$. Let $\sigma_\textbf{p}\in \sym{n+1}$ be the permutation associated to $\textbf{p}$. Then, the permutation associated to the segment $\tupp{m}{M}$ ``moves" the integer $\sigma_\textbf{p}(m)$ to position $M+1$ in the one-line notation, which will remain in that position since $\textbf{q}$ does not contain letters $M$. But since $\textbf{p}\in [m+1,M]^*$, from Lemma \ref{lemma:fixedelts} we have $\sigma_\textbf{p}(m)=m$, and so $\sigma(M+1)=\sigma_{\textbf{p}}(m)=m$. Conversely, if $\sigma(M+1)=m$ and $\textbf{a}\in \atoms{\sigma}$, then there must exist a letter $m$ that appears before a letter $M$ in $\textbf{a}$; one responsible to take the integer $m$ out of position $m$ in the one-line notation, and the other to put $m$ in position $M+1$. From condition 4 of Theorem \ref{TheoremTenner}, these letters must form a segment $\tupp{m}{M}$ in $\textbf{a}$. Moreover, these letters $m$ and $M$ must correspond to the leftmost and rightmost occurrences of $m$ and $M$ in $\textbf{a}$, respectively, also a consequence of condition 4 of Theorem \ref{TheoremTenner}. Thus, we have $\textbf{a}=\textbf{p}\,\tupp{m}{M}\,\textbf{q}$ for some $\textbf{p}\in [m+1,M]^*$ and $\textbf{q}\in [m,M-1]^*$.\par 
With similar arguments one proves that if there is a letter $M$ proceeded by a letter $m$ in $\textbf{a}$, then $\sigma(m)=M+1$ and condition 2 holds.
\end{proof}
It is worth mention that the previous result is a generalization of Lemma 21 of \cite{OEC}. 
\begin{example}
	Consider the permutation $\sigma= (9\ 2\ 10\ 1\ 8\ 3)(4\ 7)(5\ 6)\in \sym{10}$. The minimum and maximum non-fixed points of $\sigma$ are $1$ and $10$, respectively. Figure \ref{diagrams_sigma} depicts the line diagrams of the one-element commutation classes of $\sigma$. 
	\begin{figure}[h]
		\centering
		\includegraphics[scale=0.7]{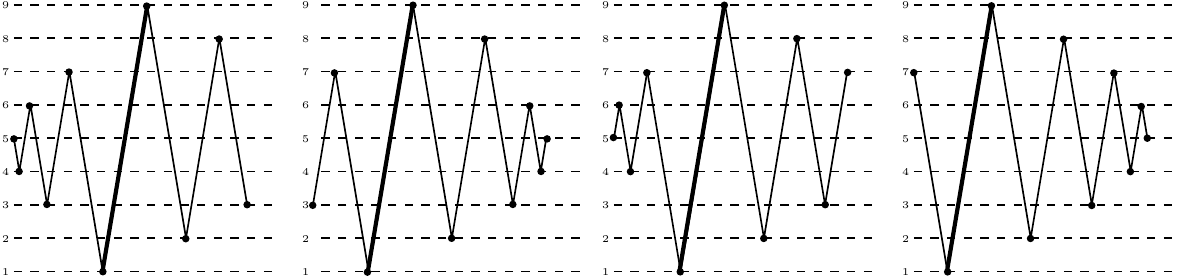}
		\caption{Line diagrams for the one-element commutation classes of $\sigma$.}
		\label{diagrams_sigma}
	\end{figure}
	From the previous lemma, the fact that $\sigma(1)=10$ implies that every one-element commutation class for $\sigma$ contains an occurrence of the segment $\unn{1}\ovv{9}$, which can be checked by looking at their line diagrams.  
	\label{ex:segs} 
\end{example}

\subsection{Oscillations}
 In this subsection we will prove Theorem \ref{thm1} for a certain subset of permutations. 
\begin{definition}
	Let $\textbf{a}\in [n]^*$ be a word formed by consecutive integers with $\pv{\textbf{a}}=i_1i_2\cdots i_r$. We say that $\textbf{a}$ is an \emph{oscillation} if the sequence $(|i_j-i_{j+1}|)_{j\in [r-1]}$ is monotone (in the broad sense).  
\end{definition}
Equivalently, a word $\textbf{a}$ is an oscillation if the lengths of the line segments in its line diagram form a sequence that is either weakly increasing or weakly decreasing, starting from its leftmost segment. Figure~\ref{osc} depicts the diagram of an oscillation. 
	\begin{figure}[h]
	\centering
	\includegraphics[scale=0.7]{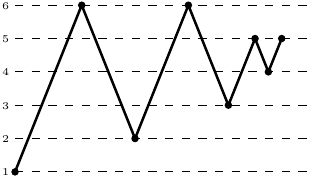}
	\caption{The line diagram of an oscillation.}
	\label{osc}
\end{figure}
In contrast, all line diagrams in Figure \ref{diagrams_sigma} represent words that are not oscillations. We are going to consider the empty word as an oscillation.  \par 

The reason we introduced this type of words  is that either every word in \(\atoms{\sigma}\) is an oscillation or none of them are, a property that will be very useful to obtain the upper bound for $|\atoms{\sigma}|$. To establish this, we will need an alternative characterization of oscillations.  \par

\begin{lemma}
	Let $\sigma \in \sym{n+1}$ with $m$ and $M+1$ the minimum and maximum non-fixed points of $\sigma$, respectively. Then, $\textbf{a}\in \atoms{\sigma}$ is an oscillation if and only if either $m$ or $M$ is an endpoint of $\textbf{a}$.  
	\label{endpoints}
\end{lemma}
\begin{proof}
	Suppose that \textbf{a} is an oscillation and let $\pv{\textbf{a}}=i_1i_2\cdots i_r$. From Lemma \ref{lemma:endpoints}, $\textbf{a}$ contains a segment $\unn{m}\,\ovv{M}$ or a segment $\ovv{M}\,\unn{m}$. Suppose that the later case holds (the other case is analogous). If $\ovv{M}\,\unn{m}$ does not contain any endpoint of $\textbf{a}$, then $\textbf{a}=\textbf{p}\,\unn{i}\,\tdownn{M}{m}\,\ovv{j}\,\textbf{q}$ for some words $\textbf{p}$ and $\textbf{q}$, and $i,j\in [m,M]$. Since $\textbf{a}$ is an oscillation we have either $M-i\geq M-m\geq j-m$ or $M-i\leq M-m \leq j-m$. However, $|i_j-i_{j+1}|\leq M-m$ for all $j\in [r-1]$, and thus the only option we have is  $i=m$ and $j=M$. But then, $\textbf{a}$ is not a reduced word because the factor $\unn{m}\,\tdownn{M}{m}\,\ovv{M}$ is not reduced, which contradicts our assumption. Therefore, the segment $\ovv{M}\,\unn{m}$ must contain an endpoint.  \par 
	Conversely, suppose that $\textbf{a}$ is a word with left endpoint the letter $m$ (the other cases are analogous). From Lemma \ref{lemma:endpoints}, $\textbf{a}=\tupp{m}{M}\,\textbf{q}$ for some word $\textbf{q}\in [m,M-1]^*$. For $\textbf{a}$ to be reduced, Theorem \ref{TheoremEquiv2} implies that $\vv{\textbf{a}}$ is weakly increasing and $\p{\textbf{a}}$ decreasing. Hence, $\textbf{a}$ is an oscillation.
\end{proof}
Then, we have the following.
\begin{proposition}
	Let $\sigma \in \sym{n+1}$ and $\textbf{a}\in \atoms{\sigma}$. If $\textbf{a}$ is an oscillation, then all one-element commutation classes of $\sigma$ are oscillations.   \label{osc2}
\end{proposition}
\begin{proof}
	Assume that $m$ and $M+1$ are the minimum and maximum non-fixed points of $\sigma$, respectively. Since $\textbf{a}$ is an oscillation, from the previous lemma we have that at least one of its endpoints must be a letter $m$ or $M$. Assume that the left endpoint of $\textbf{a}$ is a letter $m$ (the other cases are analogous). Then, we can write $\textbf{a}=\tupp{m}{M}\,\textbf{q}$ for some word $\textbf{q}\in [m,M-1]^*$ with $\sigma(M+1)=m$, by Lemma \ref{lemma:endpoints}. \par 
	If $\textbf{q}$ is empty, then $\textbf{a}=m(m+1)\cdots M$ which is a word with no commutations nor 3-braid relations available. Thus, $R(\sigma)=\{\textbf{a}\}$ and the result follows. \par 
	If $\textbf{q}$ is not empty, from Theorem \ref{TheoremEquiv2} we can write  $\textbf{a}=\tupp{m}{M}\,\unn{i}\,\textbf{q}'$ for some $i\in [m,M-1]$ and $\textbf{q}'\in [i+1,M-1]^*$. Since $\textbf{q}'$ does not contain letters $m$ nor $m-1$, the image $\sigma(m)$ is determined by the permutation associated to the left factor $\tupp{m}{M}\,\unn{i}$ of $\textbf{a}$. The possible values for $\sigma(m)$ are $M+1$ (if $i=m$) and $m+1$ (if $i>m$). These values will be important in order to arrive to a contradiction more ahead. Now suppose that $\textbf{b}\in \atoms{\sigma}$ is not an oscillation. Then, neither $m$ nor $M$ are endpoints of $\textbf{b}$, by Lemma \ref{endpoints}. Since $\textbf{b}$ must contain the segment $\tupp{m}{M}$, we can write this word as 
	$$\textbf{b}=\textbf{p}\,\ovv{j}\,\unn{m}\,\ovv{M}\,\unn{k}\,\textbf{q}'',$$  for some words $\textbf{p}$ and $\textbf{q}''$, and some integers $j$ and $k$. From Theorem \ref{TheoremTenner}, the fact that $\p{\textbf{b}}$ is a wedge and $\vv{\textbf{b}}$ is a vee implies that $\textbf{p}\in [m+1,j-1]^*$ and $\textbf{q}''\in [k+1,M-1]$. Moreover, we must have $k \neq m$; otherwise $\textbf{b}$ would contain two vales $m$, one of which would need to be an endpoint of $\textbf{b}$ by condition 5 of Theorem \ref{TheoremTenner}, implying that $\textbf{b}$ would be an oscillation by Lemma \ref{endpoints}. Using similar arguments, we also have that $j\neq M$. From Lemma \ref{lemma:fixedelts}, since $\textbf{p}\in [m+1,j-1]^*$, the permutation associated to $\textbf{p}$ fixes $j+1$. But then, the permutation associated to the segment $\tdownn{j}{m}$ is going to ``move" the integer $j+1$ to position $m$ of the one-line notation, which will remain there since $\textbf{q}''$ does not contain letters $m$. These means that $\sigma(m)=j+1$ with $m<j<M$, contradicting the fact that $\sigma(m)\in \{m+1,M+1\}$. Therefore, $\textbf{b}$ must also be an oscillation.
\end{proof}
The previous result is interesting on its own as it forces the one-element commutation classes for the same permutation to have something strong in common. Moreover, it allows us to divide our study in two cases: permutations whose one-element commutation classes are oscillations and permutations whose one-element commutation classes are not oscillations. For the first case, it remains to show the following.

\begin{lemma}
	Let $\sigma \in \sym{n+1}$ with $m$ and $M+1$ the minimum and maximum non-fixed points of $\sigma$, respectively. Then,
	\begin{enumerate}
		\item $\sigma$ contains at most one oscillation with left/right endpoint the letter $m$.
		\item $\sigma$ contains at most one oscillation with left/right endpoint the letter $M$.  
	\end{enumerate}
	\label{endpoints2}
\end{lemma}
\begin{proof}
	We use induction on the length of $\sigma$ to prove both conditions. If $\ell(\sigma)=1$, then $\sigma$ contains only one reduced word and result follows. \par 
	Let $\sigma \in \sym{n+1}$ and suppose that conditions 1 and 2 holds for every permutation $\pi\in \sym{n+1}$ satisfying $\ell(\pi)< \ell(\sigma)$. Consider  $\textbf{a},\textbf{b}\in \atoms{\sigma}$ oscillations where both words have left endpoint the letter $m$ (the other cases are analogous). Our goal is to show that $\textbf{a}=\textbf{b}$. From Lemma \ref{lemma:endpoints}, we can write these words as
	\begin{align*}
	\textbf{a}&=\tupp{m}{M}\textbf{q}\\
	\textbf{b}&=\tupp{m}{M}\textbf{q}',
	\end{align*} for some words $\textbf{q},\textbf{q}'\in [m,M-1]^*$. Note that these words $\textbf{q}$ and $\textbf{q}'$ are associated to the same permutation $\pi\in \sym{n+1}$ with $\ell(\pi)<\ell(\sigma)$. If $\pi$ is the identity permutation, then $\textbf{q}$ and $\textbf{q}'$ are the empty word and the result follows. Otherwise, since $\textbf{q},\textbf{q}'\in [m,M-1]^*$ and their left endpoint is the letter $M-1$, we have that $\textbf{q}$ and $\textbf{q}'$ are oscillations in $\atoms{\pi}$, by Lemma \ref{endpoints}. Using the induction hypothesis, there is at most one oscillation in $\atoms{\pi}$ with such left endpoint. Therefore, $\textbf{q}=\textbf{q}'$ and we have $\textbf{a}=\textbf{b}$.  
\end{proof}

We are now in condition to prove Theorem \ref{thm1} for permutations that contain oscillations as one-element commutation classes.

\begin{theorem}
	Let $\sigma\in \sym{n+1}$ such that $\atoms{\sigma}$ contains an oscillation. Then, $|\atoms{\sigma}|\leq 4$.   \label{maxosc}
\end{theorem}
\begin{proof}
	Assume that $m$ and $M+1$ are the minimum and maximum non-fixed points of $\sigma$, respectively. Since $\atoms{\sigma}$ contains an oscillation, every one-element commutation class for $\sigma$ must also be an oscillation, by Proposition \ref{osc2}. Thus, from Lemma \ref{endpoints}, every element in $\atoms{\sigma}$ satisfies at least one of the following conditions:
	\begin{itemize}
		\item $m$ is the right endpoint,
		\item $M$ is the right endpoint,
		\item $m$ is the left endpoint,
		\item $M$ is the left endpoint.
	\end{itemize}
	But from the previous lemma, there is at most one word in $\atoms{\sigma}$ for each of the previous four conditions. Therefore, $|\atoms{\sigma}|\leq 4$. 
\end{proof}

\subsection{The non-oscillation case and the proof of Theorem \ref{thm1}}
It remains to extend the previous result for permutations which do not contain oscillations as one-element commutation classes. Nevertheless, what we know about oscillations will be used as well. As might be expected, one-element commutation classes that are not oscillations exhibit a structure that depends on their associated permutation.   
\begin{lemma}
	Let $\sigma \in \sym{n+1}$ with $m$ and $M+1$ the minimum and maximum non-fixed points of $\sigma$, respectively. If $\textbf{a}\in \atoms{\sigma}$ is not an oscillation, then $\sigma$ does not contain oscillations. Moreover, there are integers $i\neq m$ and $j\neq M$ that do not depend on $\textbf{a}$ such that
	\begin{enumerate}
		\item If $\sigma(M+1)=m$, then $\textbf{a}=\textbf{p}\,\ovv{j}\,\unn{m}\,\ovv{M}\,\unn{i}\,\textbf{q}$ where $\textbf{p} \in [m+1,j-1]^*$, $\textbf{q}\in [i+1,M-1]^*$, $\sigma(m)=j+1$ and $\sigma(i)=M+1$.
		\item If $\sigma(m)=M+1$, then $\textbf{a}= \textbf{p}\,\unn{i}\,\ovv{M}\,\unn{m}\,\ovv{j}\,\textbf{q}$ where $\textbf{p} \in [i+1,M-1]^*$, $\textbf{q}\in [m+1,j-1]^*$, $\sigma(j+1)=m$ and $\sigma(M+1)=i$.  
	\end{enumerate} \label{nonosc}
\end{lemma}

\begin{proof}
From Proposition \ref{osc2}, since $\textbf{a}$ is not an oscillation there cannot exist any oscillations in $\atoms{\sigma}$. We prove only condition 1, since the proof for condition 2 is analogous. From the proof of Proposition \ref{osc2}, we have that $\textbf{a}$ can be written as $$\textbf{a}=\textbf{p}\,\ovv{j}\,\unn{m}\,\ovv{M}\,\unn{i}\,\textbf{q},$$ where $\textbf{p} \in [m+1,j-1]^*$, $\textbf{q}\in [i+1,M-1]^*$, $j\neq M$ and $i\neq m$. We have also shown that $\sigma(m)=j+1$. To prove that $\sigma(i)=M+1$, just note that the permutation associated to the segment $\tdownn{M}{i}$ ``moves" the integer $M+1$ to position $i$ in the one-line notation, which will remain there since $\textbf{q}$ does not contain letters $i$ nor $i-1$. Hence, $\sigma(i)=M+1$. 
\end{proof}
Returning to Example \ref{ex:segs}, we have that $\sigma(1)=8$, $\sigma(10)=1$ and $\sigma(2)=10$. Since $\atoms{\sigma}$ contains a non-oscillation, from condition 1 of the previous lemma, every one-element commutation class for $\sigma$ must contain the factor $\ovv{7}\,\unn{1}\,\ovv{9}\,\unn{2}$, which can be checked by looking to their line diagrams. \par 
Note that if \(\sigma \in \sym{n+1}\) satisfies \(\sigma(M+1) = m\), then \(\sigma^{-1}\) satisfies \(\sigma^{-1}(m) = M+1\), with \(m\) and \(M+1\) being the minimum and maximum non-fixed points of \(\sigma^{-1}\) as well. This gives a bijection between the sets $\{\sigma \in \sym{n+1}:\sigma(M+1)=m\}$ and $\{\sigma \in \sym{n+1}:\sigma(m)=M+1\}$, which are the sets of permutations that might have one-element commutation classes, by Lemma \ref{lemma:endpoints}. We also have a bijection between $\atoms{\sigma}$ and $\atoms{\sigma^{-1}}$, which is given by \(\textbf{a} \mapsto \textbf{a}^r\), where $\textbf{a}^r$ is the word $\textbf{a}$ written backwards. Therefore, it suffices to prove Theorem \ref{thm1} for those permutations \(\sigma\) such that \(\sigma(M+1) = m\). For that, we will consider two cases depending on the values of $i$ and $j$ mentioned in the previous lemma.  
\begin{proposition}
	Let $\sigma \in \sym{n+1}$ with $m$ and $M+1$ the minimum and maximum non-fixed points of $\sigma$, respectively. Suppose that $|\atoms{\sigma}|>0$, $\sigma(M+1)=m$, $\sigma(m)=j+1$ and $\sigma(i)=M+1$ for some $i\neq m$ and $j\neq M$. If $\atoms{\sigma}$ contains a non-oscillation and $j<i$, then $|\atoms{\sigma}|=1$.  
	\label{p1}
\end{proposition}
\begin{proof}
	Suppose that $\textbf{a}\in \atoms{\sigma}$ is not an oscillation, which we know exists by hypothesis. From the previous lemma, there is no oscillation in $\atoms{\sigma}$ and $\textbf{a}$ can be written as $$
		\textbf{a}=\textbf{p}\, \ovv{j}\, \unn{m}\, \ovv{M}\, \unn{i}\, \textbf{q},$$
	where \( \textbf{p} \in [m+1, j-1]^* \) and \( \textbf{q} \in [i+1, M-1]^* \). Since $\textbf{a}$ is a word formed by consecutive integers, if $\textbf{p}$ is not empty, then its right endpoint is the letter $j-1$. But $\textbf{p}\in [m+1,j-1]^*$, and so $j-1$ is its maximum value, implying that $\textbf{p}$ is an oscillation, by Lemma \ref{endpoints}. Similarly, if $\textbf{q}$ is not empty, then it is also an oscillation given the fact that its left endpoint (the letter $i+1$) is its minimum value. Figure \ref{dp_1} depicts the line diagram of $\textbf{a}$, where the blue, red and green parts of the diagram represents the words $\textbf{p}$, $\ovv{j}\, \unn{m}\, \ovv{M}\, \unn{i}$ and $\textbf{q}$, respectively. Denoting by $\gamma$ the permutation associated to the factor $\ovv{j}\, \unn{m}\, \ovv{M}\, \unn{i}$, we have that either $\gamma=(j+1\ m)(j+~1\ j+2\cdots i)(i\ M+1)$ if $i>j+1$, or $\gamma=(j+1\ m)(j+1\ M+1)$ if $i=j+1$. This will be important for the next step. \par 
	\begin{figure}[h]
		\centering
		\includegraphics[scale=0.6]{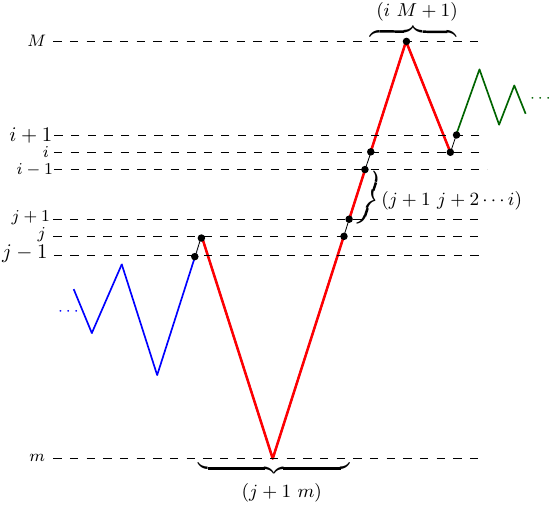}
		\caption{Line diagram of the word $\textbf{p}\, \ovv{j}\, \unn{m}\, \ovv{M}\, \unn{i}\, \textbf{q}$ when $j<i$.}
		\label{dp_1}
	\end{figure} 

Let $\sigma_{\textbf{p}},\sigma_{\textbf{q}}$ be the permutations associated to the words $\textbf{p}$ and $\textbf{q}$, respectively. We claim that the set of non-fixed points of $\sigma_{\textbf{p}},\gamma$ and $\sigma_{\textbf{q}}$ are pairwise disjoint. From Lemma~\ref{lemma:fixedelts}, and the alphabets in which $\textbf{p}$ and $\textbf{q}$ belong, we have that the non-fixed points of $\sigma_{\textbf{p}}$ and $\sigma_{\textbf{q}}$ are contained in the sets $[m+1,j]$ and $[i+1,M]$, respectively, being these sets disjoint since $j< i$. On the other hand, the set of non-fixed points of $\gamma$ is contained in $\{m,M+1\}\cup[j+1,i]$, which is disjoint from the previous two sets, thus proving the claim. \par 
Now suppose that $\textbf{b}\in \atoms{\sigma}$ is also a one-element commutation class of $\sigma$. From the previous lemma, $$\textbf{b}=\textbf{p}'\,\ovv{j}\,\unn{m}\,\ovv{M}\,\unn{i}\,\textbf{q}',$$ for some words $\textbf{p}' \in [m+1,j-1]^*$ and $\textbf{q}'\in [i+1,M-1]^*$. Moreover, the words $\textbf{p}'$ and $\textbf{q}'$ must be oscillations, as we have seen before. If we prove that $\textbf{p}$ and $\textbf{p}'$ (resp. $\textbf{q}$ and $\textbf{q}'$) have the same associated permutation, since their right endpoint (resp. left endpoint) is their maximum value (resp. minimum value), from Lemma \ref{endpoints2} we would have that $\textbf{p}=\textbf{p}'$ (resp. $\textbf{q}=\textbf{q}'$), implying that $\textbf{a}=\textbf{b}$.  To see that, let $\sigma_{\textbf{p}'}$ and $\sigma_{\textbf{q}'}$ be the permutations associated to $\textbf{p}'$ and $\textbf{q}'$, respectively. Then, using the same arguments as before, we can conclude that the sets of non-fixed points of $\sigma_{\textbf{p}'}$ and $\sigma_{\textbf{q}'}$  are contained in $[m+1,j]$ and $[i+1,M]$, respectively. Since $\textbf{a}$ and $\textbf{b}$ are reduced words for the same permutation, we have $\sigma_{\textbf{p}}\gamma \sigma_{\textbf{q}}=\sigma_{\textbf{p}'}\gamma \sigma_{\textbf{q}'}$, being the images $\{\sigma(m+1),\sigma(m+2),\ldots ,\sigma(j)\}$ in both products completely determined by the permutations $\sigma_{\textbf{p}}$ and $\sigma_{\textbf{p}'}$. Therefore, we must have $\sigma_{\textbf{p}}=\sigma_{\textbf{p}'}$, and consequently  $\sigma_{\textbf{q}}=\sigma_{\textbf{q}'}$, which concludes the proof.
\end{proof}

\begin{proposition}
	Let $\sigma \in \sym{n+1}$ with $m$ and $M+1$ the minimum and maximum non-fixed points of $\sigma$, respectively. Suppose that $|\atoms{\sigma}|>0$, $\sigma(M+1)=m$, $\sigma(m)=j+1$ and $\sigma(i)=M+1$ for some $i\neq m$ and $j\neq M$. If $\atoms{\sigma}$ contains a non-oscillation and $j\geq i$, then there is $\pi\in \sym{n+1}$ such that $|\atoms{\sigma}|\leq |\atoms{\pi}|$ and $\ell(\pi)<\ell(\sigma)$.   \label{p2}
\end{proposition}
\begin{proof}
	The strategy is to construct an injective map from \(\atoms{\sigma}\) to \(\atoms{\pi}\), where \(\pi \in \sym{n+1}\) satisfies \(\ell(\pi) < \ell(\sigma)\). This will be done by removing a certain factor from each word in $\atoms{\sigma}$. \par 
	From Lemma \ref{nonosc}, any non-oscillation $\textbf{a}\in \atoms{\sigma}$ can be written as $$\textbf{a}=\textbf{p}\,\ovv{j}\,\unn{m}\,\ovv{M}\,\unn{i}\,\textbf{q},$$ for some words $\textbf{p} \in [m+1,j-1]^*$ and $\textbf{q}\in [i+1,M-1]^*$. Since $j\geq i$ and $i>m$, we have that $\textbf{a}$ contains the factor $\ovv{i-1}\,\tupp{m}{M}\,\unn{i}$. Figure \ref{dp_2} depicts the possible line diagrams of $\textbf{a}$, where the blue, red and green parts of the line diagram correspond to the words $\textbf{p}\,\ovv{j}\,\unn{i}$, $\ovv{i-1}\,\tupp{m}{M}\,\unn{i}$ and $\textbf{q}$, respectively. On the left, we have the case where \( j > i \), and on the right, the case where \( j = i \). Note that for the latter case, $\ovv{j}\,\unn{i}$ corresponds to the single letter $i$. Denoting by $\gamma$ the permutation associated to the factor $\ovv{i-1}\,\tupp{m}{M}\,\unn{i}$, we have $\gamma=(i\ m)(i\ M+1)$. Thus, $\sigma=\sigma_{\textbf{p}} (j+1\cdots i)\gamma \sigma_{\textbf{q}}$, where $\sigma_{\textbf{p}},\sigma_{\textbf{q}}\in \sym{n+1}$ are the permutations associated to the words $\textbf{p}$ and $\textbf{q}$, respectively.
	\begin{figure}[h]
		\centering
		\includegraphics[scale=0.7]{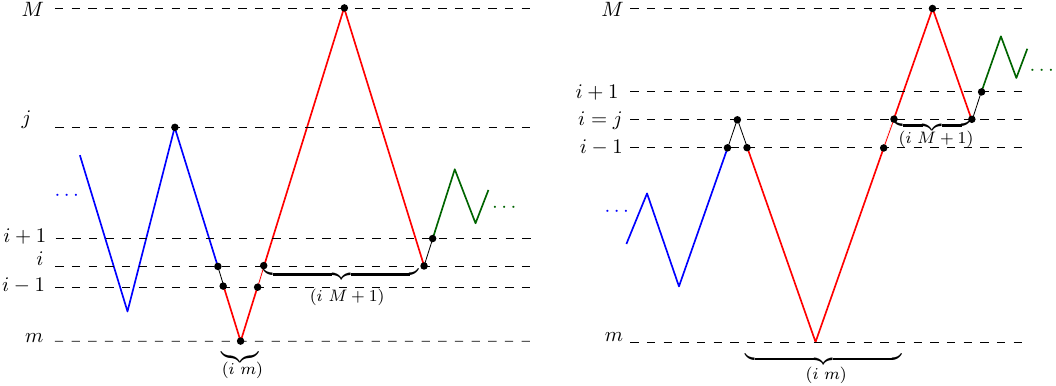}
		\caption{Possible line diagram of the word $\textbf{p}\, \ovv{j}\, \unn{m}\, \ovv{M}\, \unn{i}\, \textbf{q}$ when $j\geq i$.}
		\label{dp_2}
	\end{figure}
    If we remove the factor \(\ovv{i-1}\tupp{m}{M}\unn{i}\) from \(\textbf{a}\), which corresponds to the removal of the red part of the line diagram, we obtain a word (call it $\textbf{a}'$) formed by consecutive integers, since the left endpoint of \(\textbf{p}\,\ovv{j}\,\unn{i}\) and the right endpoint of \(\textbf{q}\) are \(i\) and \(i+1\), respectively. This word can be written as $\textbf{p}(j\, j-1\cdots i)\textbf{q}$, if $j>i$ or $\textbf{p}\ i\ \textbf{q}$ if $j=i$ (see Figure \ref{dp_3}). \par 
    \begin{figure}[h]
    	\centering
    	\includegraphics[scale=0.6]{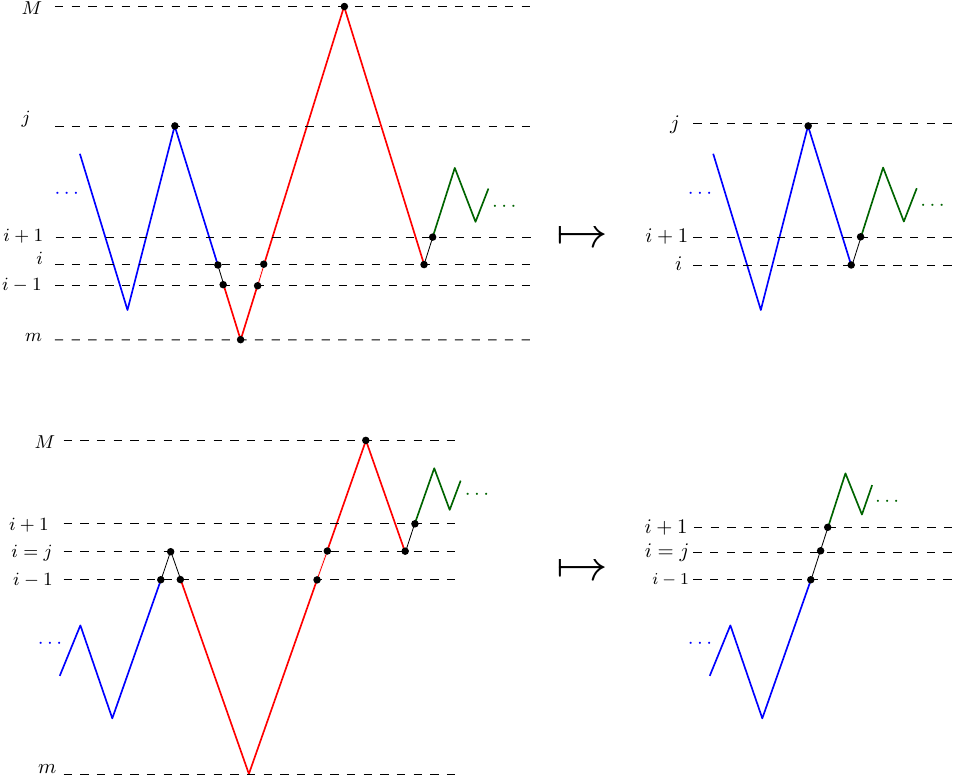}
    	\caption{Line diagrams of $\textbf{a}$ and $\textbf{a}'$.}
    	\label{dp_3}
    \end{figure}
	The fact that $\textbf{a}$ is reduced implies that it cannot contain any factor with repeated or symmetric segments, by Theorem \ref{TheoremEquiv2}. Removing the factor \(\ovv{i-1}\,\tupp{m}{M}\,\unn{i}\), we still obtain a word formed by consecutive integers satisfying the same property. Thus, $\textbf{a}'\in \atoms{\pi}$ for some $\pi \in \sym{n+1}$, by Theorem \ref{TheoremEquiv2}. Moreover, $\ell(\pi)<\ell(\sigma)$.\par 
	Our next goal is to show that $\pi$ does not depend on the choice of $\textbf{a}$. Notice that $\textbf{a}'$ is a reduced word for the permutation $\pi=\sigma_{\textbf{p}}(j+1\,j\cdots i)\sigma_{\textbf{q}}$. Since $\textbf{q}\in [i+1,M-1]^*$, from Lemma \ref{lemma:fixedelts}, the set of non-fixed points of $\sigma_\textbf{q}$ is contained in $[i+1,M]$, which is disjoint from that of $\gamma$, implying that $\sigma_\textbf{q}$ and $\gamma$ commute. Since $\sigma=\sigma_{\textbf{p}}(j+1\,j\cdots i)\gamma \sigma_\textbf{q}$, we have $\sigma \gamma^{-1}=\pi$. Note that $\gamma$ depends exclusively on $\sigma(i)$, $\sigma(m)$ and $\sigma(M+1)$, which consequently depends on $\sigma$. Thus, $\pi$ only depends on $\sigma$, as well. \par
	We are now in conditions to use the strategy mentioned in the beginning of the proof. Define the following map: 
	\begin{align*}
		\varphi: \atoms{\sigma}&\to \atoms{\sigma\gamma^{-1}}\\ \textbf{a}=\textbf{p}\ovv{j}\tupp{m}{M}\unn{i}\textbf{q}&\mapsto \varphi(\textbf{a})=\begin{cases}
		\textbf{p}(j\,j-1\cdots i)\textbf{q}, & \text{ if } j>i\\
		\textbf{p}\ i\  \textbf{q}, & \text{ if }j=i.
		\end{cases}.
	\end{align*}
	We have already shown that $\varphi$ is well-defined. To show that it is injective, suppose that $\textbf{a},\textbf{b}\in \atoms{\sigma}$ such that $\textbf{a}\neq \textbf{b}$. From Lemma \ref{nonosc}, $\textbf{a}=\textbf{p}\,\ovv{j}\,\tupp{m}{M}\,\unn{i}\,\textbf{q}$ and $\textbf{b}=\textbf{p}'\,\ovv{j}\,\tupp{m}{M}\,\unn{i}\textbf{q}'$, for some words $\textbf{p},\textbf{p}' \in [m+1,j-1]^*$ and $\textbf{q},\textbf{q}'\in [i+1,M-1]^*$. Since $\textbf{a}\neq \textbf{b}$ we have either $\textbf{p}\neq \textbf{p}'$ or $\textbf{q}\neq \textbf{q}'$. Given the restrictions of the words $\textbf{p},\textbf{p}',\textbf{q}$ and $\textbf{q}'$, we have $\varphi(\textbf{a})\neq\varphi(\textbf{b})$, showing that $\varphi$ is injective. Therefore, $|\atoms{\sigma}|\leq |\atoms{\sigma \gamma^{-1}}|$. 
\end{proof}

We have now all the ingredients to prove the main result of this paper.

\begin{theorem}
	Let $n\geq 1$. Then, for all $\sigma\in \sym{n+1}$ we have $|\atoms{\sigma}|\leq 4$.   \label{bound}
\end{theorem}
\begin{proof}
	We use induction on $\ell(\sigma)$. If $\ell(\sigma)=1$, then $\sigma$ contains only one reduced word and the result follows. \par 
	Let $\sigma \in \sym{n+1}$ and suppose that $|\atoms{\pi}|\leq 4$ for every $\pi\in \sym{n+1}$ satisfying $\ell(\pi)< \ell(\sigma)$. We may assume that $|\atoms{\sigma}|>0$, otherwise the result is trivially satisfied. Consider $m$ and $M+1$ to be the minimum and maximum non-fixed points of $\sigma$, respectively. If $\atoms{\sigma}$ contains an oscillation, then $|\atoms{\sigma}|\leq 4$ by Theorem \ref{maxosc} and we are done. Otherwise, if $\textbf{a}\in \atoms{\sigma}$ and $\textbf{a}$ is not an oscillation, then there are words $\textbf{p} \in [m+1,j-1]^*$, $\textbf{q}\in [i+1,M-1]^*$ and integers $i\neq m,j\neq M$ such that
	$$\textbf{a}=\textbf{p}\,\ovv{j}\,\unn{m}\,\ovv{M}\,\unn{i}\,\textbf{q}$$ or
	$$\textbf{a}= \textbf{q}\,\,\unn{i}\,\ovv{M}\,\unn{m}\,\ovv{j}\,\textbf{p},$$ by Lemma \ref{nonosc}. It is enough to consider the first case, since the second case follows from the first one applying the reverse word. If $j<i$, then $|\atoms{\sigma}|=1<4$, by Proposition \ref{p1}. If $j\geq i$, then there is $\pi\in\sym{n+1}$ such that $|\atoms{\sigma}|\leq |\atoms{\pi}|$ with $\ell(\pi)<\ell(\sigma)$, by Proposition \ref{p2}. Using the induction hypothesis, $|\atoms{\pi}|\leq 4$ and therefore $|\atoms{\sigma}|\leq 4$.
	\end{proof}

\section{An application}
The bound obtained in Theorem \ref{bound} will help us prove a conjecture proposed in \cite{Elder} that states the following.
\begin{conjecture}[Conjecture 7.2 of \cite{Elder}]
	For all $\sigma \in \sym{n+1}$, $\displaystyle0\leq |C(\sigma)|\leq\frac{1}{2}|R(\sigma)|+1$. \label{conj}
\end{conjecture}
 Note that \( C(\sigma) \) is the disjoint union of the set of commutation classes with more than one reduced word and the set of one-element commutation classes, that is, 
 \begin{align}
 	C(\sigma)=\{[\textbf{a}]\in C(\sigma):|[\textbf{a}]|>1\}\cup \{[\textbf{a}]\in C(\sigma):|[\textbf{a}]|=1\}. \label{disjoint}
 \end{align} Clearly, $|\{[\textbf{a}]\in C(\sigma):|[\textbf{a}]|=1\}|=|\atoms{\sigma}|$. Since the set $C(\sigma)$ partitions \( R(\sigma) \), from \eqref{disjoint} we have at least $2|C(\sigma)\setminus \atoms{\sigma}|+|\atoms{\sigma}|$ different reduced words, implying that 

 \[
 2|C(\sigma)| - |\atoms{\sigma}| \leq |R(\sigma)|,
 \]
 which simplifies to 
  \begin{align}
  |C(\sigma)| \leq \frac{1}{2} \left( |R(\sigma)| + |\atoms{\sigma}| \right).
  \label{inequality}
 \end{align}

  From Theorem \ref{bound}, $|\atoms{\sigma}|\leq 4$ and we get $$|C(\sigma)|\leq \frac{1}{2}|R(\sigma)|+2,$$ which almost proves Conjecture \ref{conjecture}. For us to have a ``+1'' instead, we need the following.

\begin{proposition}
	Let $\sigma \in \sym{n+1}$ with $m$ and $M+1$ the minimum and maximum non-fixed points of $\sigma$, respectively. Suppose that $|\atoms{\sigma}|>0$ and $\tupp{m}{M},\tdownn{M}{m}\notin \atoms{\sigma}$. If $M-m\geq 4$, then there is $[\textbf{a}]\in C(\sigma)$ such that $|[\textbf{a}]|\geq 4$.  	
\end{proposition}
\begin{proof}
	Since \( |\atoms{\sigma}| > 0 \), Lemma \ref{lemma:endpoints} implies that either \( \sigma(m) = M+1 \) or \( \sigma(M+1) = m \). Assume the latter holds (the other case is analogous), and let \( \textbf{a} \in \atoms{\sigma} \). Again from Lemma \ref{lemma:endpoints}, we can write $\textbf{a}$ as
 $$\textbf{a}=\textbf{p}\,\tupp{m}{M}\,\textbf{q}$$ for some words $\textbf{p}\in [m+1,M]^*$ and $\textbf{q}=[m,M-1]^*$. By hypothesis, \( \textbf{a} \neq \tupp{m}{M} \), and thus at least one of the words \( \textbf{p} \) or \( \textbf{q} \) is not empty. If \( \textbf{p} \) is not empty (the other case is similar), the fact that \( M - m \geq 4 \) implies that  
 \[
 \textbf{a} = \textbf{p}'\textcolor{red}{(m+1)\ m(m+1)}(m+2)(m+3)(m+4)\cdots M\ \textbf{q},
 \]
 for some word \( \textbf{p}' \). Performing a braid relation using the red factor of $\textbf{a}$, we obtain the word $$\textbf{a}'=\textbf{p}'\textcolor{red}{m\  (m+1)m}(m+2)(m+3)(m+4)\cdots M\ \textbf{q}.$$ But then, the rightmost letter \( \textcolor{red}{m} \) can commute with the factor \( (m+2)(m+3)(m+4) \) in \( \textbf{a}' \), generating three additional reduced words that belong to \( [\textbf{a}'] \). These can be obtained from \( \textbf{a}' \) in the following way:
	\begin{align*}
		\textbf{a}'&=\textbf{p}'\textcolor{red}{m\  (m+1)m}(m+2)(m+3)(m+4)\cdots M\ \textbf{q}\\
		&\sim \textbf{p}'\textcolor{red}{m\  (m+1)}(m+2)\textcolor{red}{m}(m+3)(m+4)\cdots M\ \textbf{q} \\
		&\sim \textbf{p}'\textcolor{red}{m\  (m+1)}(m+2)(m+3)\textcolor{red}{m}(m+4)\cdots M\ \textbf{q}\\
		&\sim \textbf{p}'\textcolor{red}{m\  (m+1)}(m+2)(m+3)(m+4)\textcolor{red}{m}\cdots M\ \textbf{q}.
	\end{align*}
	 Therefore, $|[\textbf{a}']|\geq 4$ and we have the result.
\end{proof}
Before proceeding to the proof of Conjecture \ref{conj}, we need the following definition.
\begin{definition}[\cite{Tenner2005ReducedDA}]
	Let $\textbf{a}=i_1i_2\cdots i_l\in R(\sigma)$. For $k\geq 1$ a positive integer, the \textit{shift} of $\textbf{a}$ by $k$ is the word $$\textbf{a}^k:=(i_1+k)(i_2+k)\cdots (i_l+k)\in R((1,\ldots,k,\sigma(1)+k,\sigma(2)+k,\ldots,\sigma(n+1)+k)).$$
\end{definition}
Letting $\pi=(1,\ldots,k,\sigma(1)+k,\sigma(2)+k,\ldots,\sigma(n+1)+k)\in \sym{n+1+k}$, we have that the map $\textbf{a}\mapsto \textbf{a}^k$ is a bijection between $R(\sigma)$ and $R(\pi)$, which also extends naturally to a bijection between $C(\sigma)$ and $C(\pi)$. 

\begin{theorem}
	For all $\sigma \in \sym{n+1}$, we have $\displaystyle0\leq |C(\sigma)|\leq\frac{1}{2}|R(\sigma)|+1$.  \label{conjecture}	
\end{theorem}
\begin{proof}
	If $|\atoms{\sigma}|=0$, then every commutation class for $\sigma$ has at least two distinct elements, and thus $\displaystyle|C(\sigma)|\leq \frac{1}{2}|R(\sigma)|<\frac{1}{2}|R(\sigma)|+1$, by \eqref{inequality}. \par 
	Suppose that $|\atoms{\sigma}|>0$ and let $m$ and $M+1$ be the minimum and maximum non-fixed points of $\sigma$, respectively. If either \( \tupp{m}{M} \in \atoms{\sigma} \) or \( \tdownn{M}{m} \in \atoms{\sigma} \), then there is only one reduced word for \( \sigma \) (that is, $|R(\sigma)| = 1$) which will be one of previous, and consequently only one commutation class, implying that \( \displaystyle |C(\sigma)| = |R(\sigma)| \leq \frac{1}{2}|R(\sigma)| + 1 \). Otherwise, if \( \tupp{m}{M} \notin \atoms{\sigma} \) and \( \tdownn{M}{m} \notin \atoms{\sigma} \), then we need to consider two separate cases.\\
	\fbox{\underline{Case 1}: $M-m\geq 4$}\\
	From the previous proposition, there is $[\textbf{a}]\in C(\sigma)$ such that $|[\textbf{a}]|\geq 4$. Then, the set $C(\sigma)$ can be written as the disjoint union $$C(\sigma)=\{[\textbf{b}]\in C(\sigma): |[\textbf{b}]|>1,[\textbf{b}]\neq[\textbf{a}]\}\cup \{[\textbf{a}]\}\cup \{[\textbf{b}]\in C(\sigma):|[\textbf{b}]|=1\}.$$ Since $C(\sigma)$ partitions the set $R(\sigma)$, we have at least $2|C(\sigma)\setminus( \{[\textbf{a}]\}\cup \atoms{\sigma})|+4+|\atoms{\sigma}|$ different reduced words. Thus,
	\begin{align*}
		2(|C(\sigma)|-1-|\atoms{\sigma}|) + 4 + |\atoms{\sigma}| \leq |R(\sigma)|
		\Leftrightarrow 2|C(\sigma)|-|\atoms{\sigma}|+2\leq |R(\sigma)|,
	\end{align*}
	 implying that
	$$2|C(\sigma)|\leq |R(\sigma)|+|\atoms{\sigma}|-2.$$
	Since $|\atoms{\sigma}|\leq 4$ by Theorem \ref{bound}, we have $2|C(\sigma)|\leq |R(\sigma)|+2$ which implies the result.\\ 
	\fbox{\underline{Case 2}: $M-m<4$}\\
	Consider $\pi=(\sigma(m)-m+1,\sigma(m+1)-m+1,\ldots,\sigma(M+1)-m+1)$. Given the fact that $m$ and $M+1$ are minimum and maximum non-fixed points of $\sigma$, we have $\{\sigma(m),\sigma(m+1),\ldots,\sigma(M+1)\}=[m,M+1]$, and so $\pi\in \sym{M-m+2}$. Then, the map $\textbf{a}\mapsto \textbf{a}^{m-1}$ is a bijection between $R(\pi)$ and $R(\sigma)$, which extends naturally to a bijection between $C(\pi)$ and $C(\sigma)$. Since $M-m<4$, we have that $\pi\in \sym{k}$ with $k<6$. Therefore, the conjecture holds in this case if it holds for all permutations in $\sym{k}$ with $k<6$, which was checked computationally. 
\end{proof}

\section{Open problems and further directions}
From Theorem \ref{thm1}, a permutation can have at most four one-element commutation classes. However, we were unable to find any permutation with exactly three such classes. If such a permutation exists, it cannot be an involution \cite{OEC}. Using a computer search, we checked all permutations up to 
$n=8$ and found none with exactly three one-element commutation classes. Table \ref{table:number_of_permutations_per_number_of_oec} shows the number of permutations $\sigma \in\sym{n+1}$ up to $n=8$ such that $|\atoms{\sigma}|=i$, with $i\in \{0,1,2,3,4\}$.

	\begin{table}[h]
		\centering
		\begin{tabular}{ c | c | c | c | c | c }
			$n\setminus |\atoms{\sigma}|$ & 0 & 1& 2 & 3 & 4\\
			\hline
			1& 1 & 1 & $-$ & $-$ & $-$ \\  
			2& 1 & 4 & 1 & $-$ & $-$  \\
			3& 5 & 15 & 3 &$-$ &1 \\
			4& 53 & 52 &12 & $-$ & 3\\
			5& 496 & 181 & 34 & $-$ &9\\
			6&4326 & 594 & 97 & $-$ & 23\\
			7& 38124 & 1875 & 261 & $-$ &60\\
			8 & 3609559 & 16937 & 1890 & $-$ & 414
		\end{tabular}
	\caption{Number of permutations containing certain number of one-element commutation classes for some values of $n$.}
	\label{table:number_of_permutations_per_number_of_oec}
	\end{table}

Based on the known results and our computational data, we propose the following conjecture.
\begin{conjecture}
	For all $\sigma \in \sym{n+1}$, we have $|\atoms{\sigma}|\in \{0,1,2,4\}$.
\end{conjecture}
Another interesting problem is to characterize the permutations that admit one-element commutation classes. For involutions, such a characterization can be given in terms of pattern avoidance \cite{OEC}. Unfortunately, this approach does not extend to the general case. The closest result we have is Lemma \ref{lemma:endpoints}, which provides a necessary but not sufficient condition for a permutation to admit one-element commutation classes. \par 

Another direction is to study this problem in other Coxeter groups. As mentioned earlier, the symmetric group is a finite Coxeter group. These groups are divided into types, being the symmetric group of type $A$. For type $B$, it has been proved that its longest element has exactly two one-element commutation classes \cite{MAMEDE2022113055}. Using computer searches for both types $B$ and $D$ we found no element with more than four one-element commutation classes in type $B$ and none with more than twelve in type $D$.

\section{Acknowledgments}
The authors were financially supported by the Fundação para a Ciência e a Tecnologia (Portuguese Foundation for Science and Technology) under the scope of the projects UID/00324/2025 (https://doi.org/10.54499/UID/00324/2025) (Centre for Mathematics of the University of Coimbra). Diogo Soares was additionally supported by the Fundação para a Ciência e Tecnologia through the grant with reference number 2022.11167.BD.
%

\bibliographystyle{plain}
\bibliography{references}

\end{document}